\documentclass{amsart}
\usepackage{amsmath,amssymb, amsthm, amscd, relsize}
\input xy
\xyoption{all}

\newtheorem{theorem}{Theorem}[section]

\newtheorem{definition}[theorem]{Definition}
\newtheorem{lemma}[theorem]{Lemma}
\newtheorem{proposition}[theorem]{Proposition}
\newtheorem{example}[theorem]{Example}

\newcommand{\Hom}{\operatorname{Hom}}
\newcommand{\M}{\operatorname{M}}
\newcommand{\Ran}{\operatorname{Ran}}
\newcommand{\Span}{\operatorname{Span}}
\newcommand{\tr}{\operatorname{tr}}

\begin{document}

\title[Unitary equivalence]{Chern classes and unitary equivalence of normal matrices over topological spaces}
\author{Greg Friedman}\thanks{This work was partially supported by a grant from the Simons Foundation (\#839707 to Greg Friedman) } \author{Efton Park}
\address{Department of Mathematics, Box 298900, 
Texas Christian University, Fort Worth, TX 76129}
\email{g.friedman@tcu.edu}\email{e.park@tcu.edu}

\begin{abstract}
This paper continues the authors' work on the question of unitary equivalence of matrices with entries in the complex-valued functions of a topological space (matrices over spaces). Specifically, we here consider the question of unitary equivalence for pairs of normal matrices over a space that share a common characteristic polynomial that can be globally factored into distinct linear factors. We show that such a matrix  is diagonalizable if and only if the first Chern classes of its eigenbundles all vanish and derive as an application that all such matrices over $\mathbb{C}P^m$ are diagonalizable for $m > 1$.  Next, given a CW complex $X$ and a polynomial $\mu$ in $C(X)[\lambda]$ that globally splits into distinct linear factors,  we prove that the number of unitary equivalence classes of matrices with $\mu$ as a characteristic polynomial depends only on the space $X$ and the degree of $\mu$, and we give some estimates on how many unitary equivalence classes there can be.  In the case that $X$ is a CW complex of dimension at most three, we demonstrate a bijection between the unitary equivalence classes of $n \times n$ normal matrices with characteristic polynomial $\mu$ and elements of the group $(H^2(X))^{n-1}$.  Finally, when $X$ is a smooth manifold and we restrict to matrices with smooth entries, we construct a de Rham cohomology class whose nonvanishing is an obstruction to unitary equivalence.
\end{abstract}

\maketitle

One of the most celebrated theorems in linear algebra is the spectral theorem: every normal matrix with complex entries is diagonalizable.  More
precisely, the spectral theorem states that if $A \in \M(n, \mathbb{C})$ is normal, then there exists a unitary
matrix $U$ in $\M(n, \mathbb{C})$ such that $U^*AU$ is diagonal.  

In \cite{GP}, K. Grove and G. K. Pedersen considered the following generalization of the spectral theorem.  Suppose $X$ is a topological space, 
let $C(X)$ denote the $\mathbb{C}$-algebra of complex-valued continuous functions on $X$, 
and let $\M(n, C(X))$ be the ring of $n$-by-$n$ matrices with entries in $C(X)$.  We will refer to elements of $\M(n, C(X))$ as
matrices over $X$.  Note that one can alternately view an element $A$ in $\M(n, C(X))$ as a continuous function $A: X \rightarrow \M(n, \mathbb{C})$.
Thus we can define the adjoint $A^*$ of $A$ pointwise, and hence the notion of normality makes sense here.
Question: for which topological spaces $X$ and which normal matrices $A$ is $A$ diagonalizable over $X$?  In other words,
when does there exists a unitary $U$ in $\M(n, C(X))$ (i.e., $U^*U = UU^* = I$) such that $U^*AU$ is diagonal?  Grove and Pedersen
discovered that to guarantee diagonalizability on a reasonable class of topological spaces, one must impose an additional
condition on elements of $\M(n, C(X))$ beyond that of normality.  The additional restriction is that $A(x)$ have distinct eigenvalues of 
multiplicity one for each $x$ in $X$; we then say that $A$ is \emph{multiplicity free}.  Among other results, Grove and Pedersen proved
that if $X$ is a $2$-connected compact CW complex ($\pi_1(X)$ and $\pi_2(X)$ both trivial), then every normal multiplicity-free matrix over $X$ is 
diagonalizable (\cite{GP}, Theorem 1.4).

In \cite{FP}, the authors of this paper studied a related question.  Suppose $A$ and $B$ are normal multiplicity-free
matrices over $X$.  Under what conditions are $A$ and $B$ unitarily equivalent?  An obvious necessary condition is that 
$A$ and $B$ must have equal characteristic polynomials, but there are already examples in \cite{GP} that show this condition is not
sufficient in general.  In \cite{FP}, we constructed a cohomology class $\theta(A, B)$, living in a twisted 
cohomology group, that is a complete obstruction to unitary equivalence.  Specifically, we proved that if $A$ and $B$
are normal multiplicity-free matrices over a CW complex $X$ with the same characteristic polynomial, then $A$ and $B$ are unitarily
equivalent if and only if $\theta(A, B) = 0$. 

Let $\mu$ denote the common characteristic polynomial of $A$ and $B$, and suppose $\mu$ splits
over $C(X)$; i.e., suppose that
\[
\mu(\lambda) = \prod_{i=1}^n(\lambda - \lambda_i)
\]
for some continuous functions $\lambda_1, \lambda_2, \dots, \lambda_n$ from $X$ to the complex numbers.  In \cite{FP}, we showed
that in this case, the invariant $\theta(A, B)$ can be expressed as a direct sum of Chern classes.

In this current paper, we restrict our attention to questions involving unitary equivalence of normal multiplicity-free matrices
that have a characteristic polynomial that splits in the aforementioned way.  
\newline

In the first section of the paper,  we establish some terminology and notation for the matrices we consider. 

In Section 2, we prove that a matrix $A$ satisfying the conditions described above is diagonalizable if and only if
the first Chern classes of the eigenvector bundles of $A$ all vanish.  We then use this result to prove that such matrices
are always diagonalizable over $\mathbb{C}P^m$ for $m > 1$.  We also consider the possible values our invariant can realize
when $X$ is $S^2 \times S^2$; this employs Proposition 6.8 of \cite{FP}, which states that there is a bijection between unitary 
equivalence classes and the set of realizable diagonalizability obstructions.

In Section 3, we ask the following question: suppose $\mu$ is a polynomial of degree $n$ in $C(X)[\lambda]$ that splits over $C(X)$.
When does there exists a normal multiplicity-free matrix over $X$ whose characteristic polynomial is $\mu$?  More generally,
how many unitary equivalence classes are there of matrices with characteristic polynomial $\mu$?  We prove that the 
number of equivalence classes is independent of the choice of $\mu$, and thus we can define $\nu_n(X)$ to be the 
number of unitary equivalence classes of normal multiplicity-free matrices that have a given characteristic polynomial
that splits over $C(X)$.  We then give some estimates on the size of $\nu_n(X)$ in various situations.

In Section 4, we prove that in some situations, every possible value of our invariant is realized.  Specifically, we prove 
that if $X$ is a CW complex of dimension at most three and if $\mu$ is a multiplicity free polynomial of degree $n$ that splits over $C(X)$,
then there is a bijection between the set of unitary equivalence classes of $n\times n$ normal 
matrices with characteristic polynomial $\mu$ and elements of the group $(H^2(X))^{n-1}$.

Finally, in Section 5, we show that when $X$ is a smooth manifold and the entries of $A$ and $B$ are smooth complex-valued functions on $X$,
we can explicitly write down a closed two-form whose de Rham cohomology class is $\theta(A, B)$, up to torsion.  
\vskip 6pt

\emph{Acknowledgments}:  We have benefited from valuable conversations with Jim Fowler, George Gilbert, Scott Nollet, Ken Richardson,
 Loren Spice, and Gerard Venema.
 
\section{Preliminaries and Notation}\label{prelims}
Throughout this paper, $X$ will be a $CW$-complex, not necessarily compact.  We may and do assume that $X$ is connected, 
because we can work with each component of a disconnected space individually.   Let $A$ be an element of $\M(n, C(X))$.  As
described in the introduction, we impose the following conditions on $A$:
\begin{itemize}
\item the matrix $A$ is normal; i.e., it commutes with its pointwise complex conjugate transpose $A^*$;
\item the matrix $A$ is multiplicity-free; i.e., for each $x$ in $X$, the eigenvalues of $A(x)$ are distinct;
\item The characteristic polynomial $\mu_A = \det(A - \lambda I)$ of $A$ splits in $C(X)[\lambda]$; i.e., we can factor $\mu_A$ into
linear factors with coefficients in $C(X)$.
\end{itemize}

When $A$ satisfies all three of these conditions, then we can choose a continuous global ordering of the eigenvalues of $A$.  In other words,
there exist continuous functions $\lambda_1, \lambda_2, \dots, \lambda_n$ from $X$ to $\mathbb{C}$ such that 
$\{\lambda_1(x), \lambda_2(x), \dots, \lambda_n(x)\}$ is a complete set of eigenvalues of $A(x)$ for each $x$ in $X$.  
Another way of saying this is that the eigenvalues of $A$ always exhibit trivial monodromy around loops in $X$.  Theorem 1.6 in \cite{GL} implies
that if $A$ is multiplicity-free and $X$ is simply connected, then the characteristic polynomial $\mu_A$ of $A$ will split in $C(X)$.

When $\mu_A$ does split in $C(X)$, there is no preferred choice of ordering of its zeros (that is, the eigenvalues of $A(x)$ for each $x$ in $X$), 
so we choose one arbitrarily, and generally tacitly.  However once we have chosen an order, that order will remain fixed. Also, when two
matrices have the same characteristic polynomial, we will use the same order of eigenvalues for both matrices.
\vskip 6pt

Throughout the rest of the paper, we will assume that the three conditions listed above hold for our matrices $A$, unless we specifically
say otherwise.
\vskip 6pt

For each $1 \leq i \leq n$, define a polynomial $p_i \in C(X)[\lambda]$ by the formula
\[
p_i(\lambda) = 
\mathlarger{\prod}_{j \neq i}(\lambda_i - \lambda_j)^{-1}(\lambda - \lambda_j).
\]

Then $p_i(\lambda_i) = 1$, and $p_i(\lambda_j) = 0$ for $j \neq i$.  Therefore, by the functional calculus for normal operators,
\[\label{E: projection equation}
P_i(x):= p_i(A(x)) = \mathlarger{\prod}_{j \neq i}(\lambda_i(x) - \lambda_j(x))^{-1}(A(x) - \lambda_j(x)I) \tag{1}
\]
is the projection of $\mathbb{C}^n$ onto the $\lambda_i(x)$ eigenspace of $A(x)$. Furthermore, for each i, the family
$\{\Ran(P_i(x))\}$ forms a complex line bundle $V_i$ over $X$ (\cite{P}, Proposition 1.7.5), and if we let $\Theta^n(X)$ denote
the trivial rank $n$ vector bundle over $X$, then
\[
\bigoplus_{i=1}^n V_i \cong \Theta^n(X).
\]
Conversely, suppose we have a decomposition of $\Theta^n(X)$ into $n$ complex line bundles $V_1, V_2, \dots, V_n$,
and further suppose we have continuous functions $\lambda_1, \lambda_2, \dots, \lambda_n$ from $X$ to $\mathbb{C}$ 
with the feature that $\lambda_1(x), \lambda_2(x), \dots, \lambda_n(x)$ have distinct values for each $x$ in $X$.  Equip
$\Theta^n(X)$ with its standard Hermitian metric, let $P_i$ be the projection of $\Theta^n(X)$ onto $V_i$, and define
\[
A =\lambda_1 P_1 + \lambda_2 P_2 + \cdots + \lambda_n P_n.
\]
Note that $P_i^2 = P_i^* = P_i$ for all $1 \leq i \leq n$ and that $P_iP_j = P_jP_i = 0$ for all $i \neq j$.  Thus
\[
AA^* = |\lambda_1|^2P_1 + |\lambda_1|^2P_2 + \cdots + |\lambda_1|^2P_n = A^*A,
\]
whence $A$ is normal.  Also, the matrix $A(x)$ has eigenvalues $\lambda_1(x), \lambda_2(x), \dots, \lambda_n(x)$
that are distinct for each $x$ in $X$.
\newline

Suppose now we have an element $B$ of $\M(n, C(X))$ that satisfies our three conditions above and also has the same characteristic
polynomial as $A$.  We can write
\[
B  = \lambda_1Q_1 + \lambda_2Q_2 + \cdots + \lambda_nQ_n
\]
where $Q_i(x)$ is the projection onto the $\lambda_i$ eigenspace of $B(x)$, and each family of vector spaces $\{\Ran Q_i(x)\}$
defines a vector bundle $W_i$ over $X$.  In this setting, the cohomology class $\theta(A, B)$ from \cite{FP} that was mentioned in the introduction
can be written in terms of Chern classes (\cite{FP}, Proposition 7.1):
\[
\theta(A, B) = c_1\left(\Hom(V_1, W_1)\right) \oplus c_1\left(\Hom(V_2, W_2)\right) \oplus \cdots \oplus c_1\left(\Hom(V_n, W_n)\right).
\]

Throughout the rest of the paper, we will adopt the notation we have established in this section.

\section{Some examples}\label{examples}

\begin{proposition}\label{diagonal}
Suppose that $A \in \M(n, C(X))$ is normal and multiplicity-free and that the characteristic polynomial of $A$ splits over $C(X)$.  Choose
an ordering $\{\lambda_1, \lambda_2, \dots, \lambda_n\}$ for the eigenvalues of $A$, and let $D$ be the diagonal matrix
with diagonal entries $\lambda_1, \lambda_2, \dots, \lambda_n$.  Then
\[
\theta(D, A) = c_1(V_1) \oplus c_1(V_2) \oplus \cdots \oplus c_1(V_n).
\]
Thus $A$ is diagonalizable if and only if $V_1, V_2, \dots, V_n$ all have trivial first Chern class.
\end{proposition}

\begin{proof}
The $\lambda_i(x)$ eigenspace of $D$ is the subspace of $\mathbb{C}^n$ spanned by the vector that is 1 in the $i$th slot and zero elsewhere, and
therefore the corresponding vector subbundle is (isomorphic to) the trivial line bundle $\Theta^1(X)$.  
We note that $\Hom(\Theta^1(X), V_i) \cong V_i$ for each $1 \leq i \leq n$, whence the first claim of the proposition follows.  
Proposition 7.1 in \cite{FP} establishes the second claim.
\end{proof}

\begin{definition}\label{obstructionsdef}
The Chern classes $c_1(V_1), c_1(V_2), \dots, c_1(V_n)$ in Proposition \ref{diagonal} are called the 
\emph{obstructions to diagonalizability} for $A$.
\end{definition}

\begin{lemma}\label{relations}
For $k > 0$, the elementary symmetric polynomials $s_k$ evaluated at $c_1(V_1), c_1(V_2), \dots, c_1(V_n)$ vanish.
\end{lemma}

\begin{proof}
For each $i$, let $c(V_i)$ denote the total Chern class of $V_i$.  
Because each $V_i$ is a line bundle, we have that $c(V_i) = 1 + c_1(V_i)$.
Applying the Whitney product formula \cite[Formula 14.7]{MS},

\begin{multline*}
1 = c(\Theta^n(X)) = c\bigl(\oplus_{i=1}^n V_i\bigr) = \prod_{i=1}^n c(V_i)\\
= \prod_{i=1}^n (1+c_1(V_i)) = 1 + \sum_{k=1}^n s_k\left(c_1(V_1), c_1(V_2), \dots c_1(V_n)\right).
\end{multline*}
\end{proof}

\begin{proposition}\label{CPm}
Suppose that $A$ in $\M(n, C(\mathbb{C}P^m))$ is normal and multiplicity-free and that $m > 1$.  Then $A$ is diagonalizable.
\end{proposition}

\begin{proof}
Because $\mathbb{C}P^m$ is simply connected, the characteristic of polynomial of $A$ splits over $\mathbb{C}P^m$.
Proposition \ref{diagonal} states that we need only show that $c_1(V_i) = 0$ for $1 \leq i \leq n$.

Note that 
\[
(s_1(x_1, x_2, \dots, x_n))^2 - 2s_2(x_1, x_2, \dots, x_n) = x_1^2 + x_2^2 + \cdots x_n^2.
\]
Thus, in light of Lemma \ref{relations}, we see that
\[
(c_1(V_1))^2 + (c_1(V_2))^2 + \cdots + (c_1(V_n))^2 = 0.
\]
The cohomology ring $H^*(\mathbb{C}P^m)$ is isomorphic to $\mathbb{Z}[\alpha]/\alpha^{m+1}$ (\cite[Example 3.40]{Ha}), 
and identifying these two rings via this isomorphism, we observe that for each $i$ 
we have $c_1(V_i)=k_i\alpha$ for some integer $k_1$.   Hence
\[
0=\sum_{i=1}^n (c_1(V_i))^2=\sum_{i=1}^n (k_i\alpha)^2=\left(\sum_{i=1}^nk_i^2\right)\alpha^2\in H^4(\mathbb{C}P^m).
\]
Because $m>1$, the class $\alpha^2$ is a generator of $H^4(\mathbb{C}P^m)\cong \mathbb{Z}$, and therefore
all the integers $k_i$ are zero.  Therefore $c_1(V_i) = 0$ for all $1 \leq i \leq n$.
\end{proof}

By contrast, for $m = 1$, Example 7.2 in \cite {FP} gives examples of normal multiplicity-free matrices whose characteristic
polynomials split, but are not diagonalizable.  We will have more to say about matrices over $\mathbb{C}P^1 = S^2$ in Section \ref{low}.

\begin{example}\label{2by2}
$2\times 2$ multiplicity-free normal matrices on $S^2\times S^2$
\end{example}

We prove that there is a bi-infinite family of unitary equivalence classes of $2\times 2$ normal matrices over 
$S^2\times S^2$ for any multiplicity-free characteristic polynomial.   We begin by observing that $S^2\times S^2$ is simply connected.
Let $\alpha$ denote a generator of $H^2(S^2) \cong \mathbb{Z}$. 
Define $\bar \alpha=\alpha\times 1$ and  $\bar \beta=1\times \alpha$ in $H^2(S^2\times S^2)$. 
Then $\bar \alpha$ and $\bar \beta$ are the standard generators of $H^2(S^2\times S^2)\cong \mathbb{Z}^2$, and  
$\bar \alpha\cup\bar\beta=(\alpha\times 1)\cup(1\times \alpha)=\alpha\times \alpha$ is a generator of $H^4(S^2\times S^2)$, 
while $\bar \alpha\cup \bar \alpha=\bar \beta\cup \bar \beta=0$.  
 
Suppose that $A \in \M(2, S^2 \times S^2)$ is normal and multiplicity free.  Then
\begin{align*}
c_1(V_1) &= k_1\bar \alpha+\ell_1\bar \beta \\
c_1(V_2) &= k_2\bar \alpha+\ell_2\bar \beta
\end{align*}
for some integers $k_1$, $k_2$, $\ell_1$, and $\ell_2$.
By Lemma \ref{relations},
\[
s_1(c_1(V_1), c_1(V_2)) = c_1(V_1)+c_2(V_2)=0,
\]
whence 
\begin{align*}
0 &= c_1(V_1)+c_2(V_2)\\
&= k_1\bar \alpha+\ell_1\bar \beta+k_2\bar \alpha+\ell_2\bar \beta\\
&=( k_1+k_2)\bar \alpha+(\ell_1+\ell_2)\bar \beta,
\end{align*}
and so $k_2=-k_1$, $\ell_2=-\ell_1$. 
Next, because
\[
s_2(c_1(V_1), c_1(V_2)) = c_1(V_1)c_1(V_2)=0,
\]
we see that 
\begin{align*}
0 &= c_1(V_1)c_1(V_2)=(k_1\bar \alpha+\ell_1\bar \beta)(-k_1\bar \alpha-\ell_1\bar \beta)\\
& =(-2k_1\ell_1)\bar \alpha\cup \bar \beta. 
\end{align*}
Thus one of $k_1$ or $\ell_1$ must be $0$.  Therefore the possible Chern classes for $V_1$ and $V_2$ are
$c_1(V_1)=k\bar \alpha$ and $c_2(V_2)=-k\bar \alpha$, or else $c_1(V_1)=\ell\bar \beta$ and $c_1(V_2)=-\ell\bar \beta$. 
We will show that for any multiplicity-free characteristic polynomial on $S^2\times S^2$, we can realize all such classes for any 
integers $k$ and $\ell$. 

In Example 7.2 of \cite{FP}, we constructed examples of $2 \times 2$ matrices over $\mathbb{C}P^1\cong S^2$ with
diagonalizability obstructions $c_1(V_1) = k\alpha$, $c_1(V_2) = - c_1(V_1) = -k\alpha$.
Our construction here will generalize that construction. First, we note that every map $f: S^2\to \mathbb{C}P^1$ determines a line 
bundle over $S^2$ that can be considered a subbundle of the trivial bundle $\Theta^2(S^2)$ (just let the complex line over $x \in S^2$ be $f(x)$). 
Equivalently, this line bundle is $f^*\gamma^1$, the pullback by $f$ of the tautological line bundle $\gamma^1$ of $\mathbb{C}P^1$. 
We know that the Chern class of the tautological line bundle generates $H^2(\mathbb{C}P^2)$ \cite[Theorem 14.4]{MS}, so if we let
$\alpha$ be this generator and let $f$ be a degree $k$ map, then $c_1(f^*\gamma^1)=k\alpha$. 

Next, let $\pi: S^2\times S^2 \rightarrow S^2$ be the projection to the first copy of the 2-sphere.
Then $\pi^*f^*\gamma^1=(f\pi)^*\gamma^1$ is a line bundle over 
$S^2\times S^2$ and a subbundle of $\pi^*(\Theta(S^2)) \cong \Theta^2(S^2 \times S^2)$.
Furthermore, by naturality of the Chern classes,
\[
c_1((f\pi)^*\gamma^1)=\pi^*c_1(f^*\gamma^1)=\pi^*(k\alpha)=k\bar \alpha. 
\]
Endow $\Theta(S^2 \times S^2)$ with its standard Hermitian structure, let $P_1$ denote the projection of $\Theta(S^2 \times S^2)$ 
onto $(f\pi)^*\gamma^1$, and define $P_2 = I - P_1$.  

Now, let $\mu$ be any multiplicity-free degree two polynomial in $C(S^2\times S^2)[\lambda]$.  Designate the roots of $\mu$ as
$\lambda_1$ and $\lambda_2$;  we can do this because $S^2\times S^2$ is simply connected.
Set
\[
A = \lambda_1P_1 + \lambda_2P_2.
\]
Then $A$ is a normal multiplicity-free matrix over $S^2 \times S^2$ with characteristic polynomial $\mu$, and by Proposition \ref{diagonal},
the diagonalizablity obstructions of $A$ are $k\bar \alpha$ and $-k\bar \alpha$. A similar procedure with projection 
to the second component gives us matrices with diagonalizability obstructions $\ell\bar \beta$ and $-\ell\bar \beta$.  
Proposition 6.8 of \cite{FP} states that there is a bijection between unitary equivalence classes  and the set of realizable diagonalizability
obstructions, so this gives us the desired result.
\newline

The situation already becomes significantly more complicated in the example above if $n = 3$.   Suppose that $V_1$, $V_2$, and $V_3$ are line bundles
over $S^2 \times S^2$ and that $V_1 \oplus V_2 \oplus V_3$ is trivial.  Then
\begin{align*}
c_1(V_1) &= k_1\bar \alpha+\ell_1\bar \beta \\
c_1(V_2) &= k_2\bar \alpha+\ell_2\bar \beta \\
c_1(V_3) &= k_3\bar \alpha+\ell_3\bar \beta
\end{align*}
for some integers $k_1$, $k_2$, $k_3$, $\ell_1$, $\ell_2$, and $\ell_3$.  The equation $s_1(c(V_1), c(V_2), c(V_3)) = 0$ implies
that
\[
k_1 + k_2 + k_3 = \ell_1 + \ell_2 + \ell_3 = 0,
\]
and the equation $s_2(c(V_1), c(V_2), c(V_3)) = 0$ yields
\[
k_1\ell_2 + k_2\ell_1 + k_1\ell_3 + k_3\ell_1+ k_2\ell_3 + k_3\ell_2 = 0. \tag{$\star$}
\]
Thus $k_3$ is determined by $k_1$ and $k_2$, and $\ell_3$ is determined by $\ell_1$ and $\ell_2$.  Plugging this information
into the equation above gives us the following
necessary condition on $k_1$, $k_2$, $\ell_1$, and $\ell_2$:
\[
k_1(2\ell_1 + \ell_2) + k_2(\ell_1 + 2\ell_2) = 0.
\]
We do not know precisely which integers satisfying this equation can actually be realized
in our expressions for the Chern classes for $V_1$, $V_2$, and $V_3$.

\section{How many unitary equivalence classes?}\label{classes}

We have been considering multiplicity-free normal matrices $A$ over a $CW$-complex $X$ with the 
property that the characteristic polynomial $\mu_A$ splits over $C(X)$.  Suppose we have a polynomial $\mu$ in $C(X)[\lambda]$
that splits.  When is it true that $\mu = \mu_A$ for some multiplicity-free matrix $A$ over $X$?  Of course, any matrix 
unitarily equivalent to $A$ has the same characteristic polynomial as $A$, so the question is better phrased this way: Given such
a polynomial $\mu$, how many unitary equivalence classes of multiplicity-free normal matrices are there with characteristic polynomial
$\mu$?  We have the following somewhat surprising result:

\begin{proposition}\label{charpoly}
Suppose that $X$ is a $CW$ complex and let $\mu, \widetilde{\mu} \in C(X)[\lambda]$ be multiplicity-free polynomials that split over $C(X)$ and have 
the same degree.  Then the number of unitary equivalence classes of normal matrices over $X$ with characteristic polynomial $\mu$ is equal to the
number of unitary equivalence class of normal matrices over $X$ with characteristic polynomial $\widetilde{\mu}$.
\end{proposition}

\begin{proof} 
Suppose that $A$ is a normal matrix in $\M(n, C(X))$ that has characteristic polynomial $\mu$, and let
$\lambda_1(x), \lambda_2(x), \dots, \lambda_n(x)$ be a continuous ordering of the roots of $\mu(x)$ as $x$ ranges over $X$.
Then we can write $A$ in the form 
\[
A = \lambda_1 P_1 + \lambda_2 P_2 + \cdots + \lambda_n P_n.
\]
Let $\widetilde{\lambda}_1(x), \widetilde{\lambda}_2(x), \dots, \widetilde{\lambda}_n(x)$ be a continuous ordering of the roots of $\widetilde{\mu}(x)$. Then
\[
\widetilde{A} = \widetilde{\lambda}_1 P_1 + \widetilde{\lambda}_2 P_2 + \cdots + \widetilde{\lambda}_n P_n
\]
is a normal matrix over $C(X)$ with characteristic polynomial $\widetilde{\mu}$. Reversing the roles of $\mu$ and $\widetilde{\mu}$, we see
that $A \longleftrightarrow \widetilde{A}$ is a one-to-one correspondence between normal matrices with characteristic polynomial $\mu$
and normal matrices with characteristic polynomial $\widetilde{\mu}$.  

Now suppose that $B$ has the same characteristic polynomial as $A$ and that $B = U^*AU$ for some unitary matrix $U$
in $\M(n, C(X))$.  Then
\begin{align*}
B &= U^*(\lambda_1 P_1 + \lambda_2 P_2 + \cdots + \lambda_n P_n)U \\
&= \lambda_1 U^*P_1U + \lambda_2 U^*P_2U + \cdots + \lambda_nU^*P_nU,
\end{align*}
while
\begin{align*}
\widetilde{B} := U^*\widetilde{A}U &= U^*(\widetilde{\lambda}_1 P_1 + \widetilde{\lambda}_2 P_2 + \cdots + \widetilde{\lambda}_n P_n)U \\
&= \widetilde{\lambda}_1 U^*P_1U + \widetilde{\lambda}_2 U^*P_2U + \cdots + \widetilde{\lambda}_nU^*P_nU.
\end{align*}
Again switching the roles of $\mu$ and $\widetilde{\mu}$, we obtain the desired result.
\end{proof}

In light of Proposition \ref{charpoly}, we can make the following definition.

\begin{definition}
Let $\nu_n(X)$ be the number of unitary equivalence classes of $n\times n$ matrices over $X$ with a given multiplicity-free
characteristic polynomial that splits over $C(X)$.
\end{definition}

Thanks to Proposition 6.8 in \cite{FP}, we can alternatively define $\nu_n(X)$ to be the number of diagonalization obstructions for $n$-by-$n$
matrices over $X$.

\begin{proposition}\label{splitting}
Suppose $f: Y \rightarrow X$ is a map of $CW$ complexes that possesses a splitting $s$, i.e. a map $s:X \rightarrow Y$ such that
$fs:X \rightarrow X$ is the identity. Then $\nu_n(Y)\geq \nu_n(X)$. 
\end{proposition}

\begin{proof}
Let $A$ be an $n\times n$ multiplicity-free normal matrix over $X$ whose characteristic polynomial splits over $C(X)$, and
let $f^*A$ be the matrix on $Y$ with $(f^*A)(y)=A(f(y))$. Then $f^*A$ is normal and multiplicity free and its characteristic polynomial
splits over $C(Y)$.  The obstruction to diagonalizing $A$ over $X$ is $\oplus_{i=1}^nc_1(V_i)$, and due to the naturality of pullbacks and Chern
classes, the obstruction to diagonalizing $f^*A$ is 
\[
\bigoplus_{i=1}^nc_1(f^*V_i) = \bigoplus_{i=1}^nf^*c_1(V_i).
\]
The composition $fs$ is the identity, so $(fs)^*=s^*f^*:H^2(X) \rightarrow H^2(X)$ is also the identity, and so $f^*:H^2(X) \rightarrow H^2(Y)$ is injective. 
Thus there are at least as many realizable obstructions to diagonalizability over $Y$ as there are over $X$.
\end{proof}

\begin{proposition}\label{nuproduct}
Suppose that $X$, $Y$ are $CW$ complexes whose homology groups are finitely generated in each dimension.  Then
\[
\nu_n(X\times Y)\geq \nu_n(X)+\nu_n(Y)-1.
\]
\end{proposition}

\begin{proof}
By the K\"unneth theorem (\cite{MKat}, Theorem 60.5) the group $H^2(X)\oplus H^2(Y)$ is a direct summand of $H^2(X\times Y)$. 
In particular, the maps $\cdot \times 1:H^2(X) \rightarrow H^2(X\times Y)$ and $1\times \cdot:H^2(Y) \rightarrow H^2(X\times Y)$ are injective 
and the intersection of these two subgroups is $0$. Now we precede as in the preceding proposition.
Let $A$ be an $n\times n$ multiplicity-free normal matrix over $X$ whose characteristic polynomial splits over $C(X)$.
 Let $\pi:X\times Y\to X$ be the projection map, and define $s:X \rightarrow Y$ to be the inclusion $s(x)=(x,y_0)$, where $y_0$ is a
 fixed chosen basepoint for $Y$.  
By Theorem 61.2 in \cite{MKat}, the obstruction to diagonalizing $\pi^*A$ is
\[
\bigoplus_{i=1}^nc_1(\pi^*V_i)=\bigoplus_{i=1}^n\pi^*c_1(V_i)=\bigoplus_{i=1}^nc_1(V_i)\times 1_Y.
\]
Because $\cdot\times 1_Y$ is injective by the K\"unneth theorem, there are thus at least as many realizable obstructions to 
diagonalizability of this form in $X\times Y$ as there are obstructions to diagonalizability of the form 
$\oplus_{i=1}^nc_1(V_i)$ over $X$, which is just $\nu_n(X)$. 

This argument works equally well by interchanging the roles of $X$ and $Y$, and we know from the injectivity of the K\"unneth theorem 
that a class of the form $\oplus_{i=1}^n\alpha_i\times 1_Y$ equals a class of the form $\oplus_{i=1}^n1_X\times \beta_i$ only if all the 
$\alpha_i$ and $\beta_i$ are equal to $0$. So there is only one unitary equivalence class that can be pulled back from either $X$ or $Y$, 
and this is the class of the diagonalizable matrices, which are certainly the pullbacks of the diagonalizable matrices over $X$ and $Y$.
\end{proof}

\section{Low dimensions}\label{low}

Our goal in this section is to prove the following theorem. 

\begin{theorem}\label{lowdim}
Let $X$ be a CW complex with $\dim(X)\leq 3$, and let $\mu \in C(X)[\lambda]$ be a  multiplicity free polynomial of degree $n$ that 
splits over $C(X)$. There is a bijection between the set of unitary equivalence classes of $n\times n$ normal matrices with characteristic
polynomial  $\mu$ and elements of the group $(H^2(X))^{n-1}=\displaystyle \oplus_{i=1}^{n-1} H^2(X)$. 
\end{theorem}

We establish some notation.   Elements of $\prod_{i=1}^k \mathbb{C}P^m$ can be written in the form $(x_1, x_2, \ldots, x_k)$, where each
$x_i$ is a complex line in $\mathbb{C}^{m+1}$.  Let $\Span(x_1, x_2, \ldots, x_k)$ 
denote the span of the vectors contained in the union of the $x_i$ for $1 \leq i \leq k$.

\begin{lemma}\label{CPmanifold}
For each $k\leq m$, let $U_k$ be the open subset consisting of points $(x_1, x_2, \dots, x_k)$ in $\prod_{i=1}^k \mathbb{C}P^m$ 
such that $\Span(x_1,\ldots, x_k)$ has (complex) dimension $k$. 
Let $S \subset U_k\times \mathbb{C}P^m$ consist of those $(x_1,\ldots, x_{k+1})\in U_k\times \mathbb{C}P^m$ such that 
$x_{k+1}$ is in $\Span(x_1,\ldots, x_k)$. Then $S$ is a smooth submanifold of $U_k\times \mathbb{C}P^m$ 
of real dimension $2(mk+k-1)$. 
\end{lemma}

\begin{proof}
Throughout the following argument we will use complex coordinates and complex vector space operations for clarity and convenience.
However, when talking about differentiability, we will mean differentiability of the \emph{real} components (i.e. the real and imaginary parts)
of complex functions with respect to the underlying \emph{real} coordinates. In particular, bear in mind that all of our spaces and
manifolds of complex dimension $n$ are also regarded as real spaces of dimension $2n$. 

It suffices to prove the claim locally.  For this, we build up to an application of the implicit function theorem.  We start by choosing a point
$p$ in $S$ and then constructing a ``moving frame'' that assigns to each point in a neighborhood of $p$ in $U_k \times \mathbb{C}P^m$
a certain basis of $\mathbb{C}^m$.

Fix $p=(p_1,p_2, \dots, p_{k+1})$ in $S$, and form a coordinate chart around $p$ consisting of 
inhomogeneous (affine) coordinates as follows: Recall that for any $y \in \mathbb{C}P^m$, which is represented in homogeneous
coordinates $[y_1, y_2, \dots, y_{m+1}]$, we can choose some $j$ such that $y_j \neq 0$ and identify $y$ with the coordinates
\[
(y_1/y_j, y_2/y_j, \dots, \hat{1}_j, \dots y_{m+1}/y_j)
\]
in $\mathbb{C}^{m}$, where $\hat 1_j$ denotes that we are omitting
the $j$th coordinate, which is $1$.  If we let $C_j$ denote the open subset of points in $\mathbb{C}P^m$ with non-zero $j$th homogeneous
coordinate, then we can treat this construction as a map from $C_j$ to $\mathbb{C}^m$, and this provides a chart; this is the standard way to provide 
a manifold atlas for $\mathbb{C}P^m$. It will be useful to think of $\mathbb{C}^m$ as embedded in $\mathbb{C}^{m+1}$ as the plane 
with $j$ coordinate equal to $1$.
Following this procedure in each coordinate, we get a map $\phi$ from a neighborhood $W$ of $p$ in $U_k\times \mathbb{C}P^m$ to 
$(\mathbb{C}^{m+1})^{k+1}$. Explicitly, this map takes  $x=(x_1,x_2, \dots, x_{k+1})$  in $W$ to 
$\phi(x) = (\phi_1(x_1), \phi_2(x_2), \dots, \phi_{k+1}(x_{k+1}))$, where
\[
\phi_i(x_i) = \left(\frac{x_{i,1}}{x_{i,j(i)}},\dots, 1_{j(i)},\dots,\frac{x_{i,m+1}}{x_{1,j(i)}}\right)
\]
and $j(i)$ is the chosen non-zero coordinate for $x_i$.   Notice that each $\phi_i(x_i)$ is a point in the complex line spanned
by $x_i$ 
so 
\[
\Span(x_1, x_2, \dots, x_k) = \Span(\phi_1(x_1), \phi_2(x_2), \dots, \phi_k(x_k)).
\]
By definition of $U_k$, this span is always $k$-dimensional.  Also, because $p$ is a point in $S$, we have
$\phi_{k+1}(p)$ in $\Span(\phi_1(p), \phi_2(p), \dots, \phi_k(p))$ by definition.
Let $v_{k+1},\dots, v_{m+1}$ be a collection of vectors in $\mathbb{C}^{m+1}$ such that 
\[
\Span(\phi_1(p), \phi_2(p), \dots, \phi_k(p), v_{k+1},\dots, v_{m+1}) = \mathbb{C}^{m+1}.
\]
In other words, the $v_\ell$, for $k < \ell \leq m+1$, complete the set 
$\{\phi_i(p)\}_{i=1}^k$ to a basis of $\mathbb{C}^{m+1}$. Because the $v_\ell$ are fixed vectors, by possibly making $W$ smaller, we can take a neighborhood $W$ of $p$ in $U_k\times \mathbb{C}P^m$ such that
$\Span(\phi_1(x),\phi_2(x), \dots, \phi_k(x), v_{k+1},\dots, v_{m+1})$ is a basis of $\mathbb{C}^{m+1}$ for all $x$ in $W$. 

Next, we perform the Gram-Schmidt process as a function of $x$ to obtain orthonormal bases of $\mathbb{C}^{m+1}$ at each
$x$ in $W$ with its first $k$ coordinates in the set $\Span(\phi_1(x), \phi_2(x), \dots, \phi_k(x))$. Let
\begin{align*}
e_1(x) &= \frac{\phi_1(x)}{|\phi_1(x)|}\\
&\vdots\\
e_k(x)&= \frac{\phi_k(x) - \sum_{i=1}^{k-1}\frac{\phi_k(x)\cdot e_i(x)}{|e_i(x)|^2}e_i(x)}
{\left| \phi_k(x) - \sum_{i=1}^{k-1}\frac{\phi_k(x)\cdot e_i(x)}{|e_i(x)|^2}e_i(x)\right|}\\
e_{k+1}(x)&= \frac{v_{k+1}-\sum_{i=1}^{k}\frac{v_{k+1}\cdot e_i(x)}{|e_i(x)|^2}e_i(x)}
{\left| v_{k+1}-\sum_{i=1}^{k}\frac{v_{k+1}\cdot e_i(x)}{|e_i(x)|^2}e_i(x)\right|}\\
&\vdots\\
e_{m+1}(x)&= \frac{v_{m+1}-\sum_{i=1}^{m}\frac{v_{m+1}\cdot e_i(x)}{|e_i(x)|^2}e_i(x)}
{\left| v_{m+1}-\sum_{i=1}^{m+1}\frac{v_{k+1}\cdot e_i(x)}{|e_i(x)|^2}e_i(x)\right|}.
\end{align*}
Because (the real components of) the functions $\phi_i$ depend smoothly on $x$ by the definition of charts,
so do the (real components of) $e_i(x)$. And clearly we have
\[
\Span(e_1(x),\dots,e_j(x)) = \Span(\phi_1(x), \phi_2(x), \dots,\phi_j(x))
\]
for any $j\leq k$.  These $e_j$ constitute our ``moving frame.''

Next, we define $P$ to be the projection function that takes $x\in W$ to the projection of $\phi_{k+1}(x)$ onto the orthogonal subspace
of $\Span(\phi_1(x), \phi_2(x), \dots, \phi_k(x))$. This function is simply 
\[
P(x) = \phi_{k+1}(x)-\sum_{i=1}^k \frac{\phi_{k_1}(x)\cdot e_i(x)}{|e_i(x)|^2}e_i(x),
\]
which is also smooth. 

For $k+1\leq j\leq m+1$, let $f_j(x)=P(x)\cdot e_j(x)$, the $j$th coordinate of $P(x)$ in the basis $\{e_i(x)\}$. 
Let $F:W \rightarrow \mathbb{C}^{m+1-k}$ be the smooth function $F(x) = (f_{k+1}(x),\dots,f_{m+1}(x))$. 
Altogether, for $x$ in $W$, the image $F(x)$ represents the coordinates of a  projection of $\phi_{k+1}(x)$ to the subspace of
$\mathbb{C}^{m+1}$ orthogonal to $\Span(\phi_1(x),\phi_2(x), \dots, \phi_k(x))$,
though we have to allow the basis with respect to which the coordinates are chosen to vary with $x$.  The key to our entire construction
is the observation that $F(x)=0$ if and only if 
\[
\phi_{k+1}(x) \in \Span(\phi_1(x), \phi_2(x), \dots, \phi_k(x)),\]
which corresponds to $x\in S$. 

We next establish that the differential of $F$ has maximal (real) rank at $p$ by showing that every $v$ in 
$\Span(v_{k+1}, \dots v_{m+1})$ is in the image of $Df$.  At $p$, the orthogonal subspace to 
$\Span(\phi_1(p), \phi_2(p), \dots, \phi_k(p))$ is spanned by the fixed vectors $v_{k+1},\dots, v_{m+1}$, which are orthogonal to 
$\phi_{k+1}(p)$.  As above, let $j(k+1)$ be the index of $\phi_{k+1}(p)$ set to $1$ in our coordinate system.  For any vector $v$ in
$\Span(v_{k+1},\dots, v_{m+1})$, consider the curve $\phi_{k+1}(p) + tv$, $t\in \mathbb{R}$, $|t|$ small.
Let $J(\phi_{k+1}(p)+tv)$ be the $j(k+1)$ coordinate of $\phi_{k+1}(p)+tv$, which is non-zero for sufficiently small $t$, say $|t| < \epsilon$.
Then for $|t|<\epsilon$, we have a smooth curve
\[
\gamma(t) = \frac{\phi_{k+1}(p)+tv}{J(\phi_{k+1}(p)+tv)}
\]
with $\gamma(0) = \phi_{k+1}(p)$. Abbreviating the denominator to $J(t)$ (so $J(0)=1$), we can write this as 
$\gamma(t)=\frac{\phi_{k+1}(p)}{J(t)}+\frac{t}{J(t)}v$. But $\frac{\phi_{k+1}(p)}{J(t)}$ remains in 
$\Span(\phi_1(p), \phi_2(x), \dots, \phi_k(p))$ (because $\phi_{k+1}(p)$ is), while $\frac{t}{J(t)}v$ is a scalar multiple of $v$. 
It follows from the definition that $F(\phi_1(p),\phi_2(p), \dots, \phi_k(p), \gamma(t)) = \frac{t}{J(t)}v$. Furthermore,
\[
\frac{d}{dt}\left(\frac{t}{J(t)}\right)=\frac{J(t)-tJ'(t)}{(J(t))^2}.
\]
At $t=0$, this is $1$. Altogether then, 
\[
\frac{d}{dt}(F(\phi_1(p), \phi_2(p) \dots, \phi_k(p), \gamma(t)))|_{t=0} = v,
\]
and so each $v$ is in the image of the differential of $F$ at $p$. Therefore $F$ has real rank $2(m+1-k)$.

It now follows from the implicit function theorem for manifolds \cite[Corollary II.7.4]{Br} that, in a neighborhood of $p$, 
the zero set of $F$ (which we have established is $S$) is a smooth submanifold. But $p\in S$ was arbitrary, so $S$ is a smooth
submanifold of $U_k\times \mathbb{C}P^m$.  As the source of $F$ has real dimension $2(mk+m)$ and the target is 
$\mathbb{C}^{m+1-k} \cong \mathbb{R}^{2(m+1-k)}$, the real dimension of $S$ is 
\[
2(mk+m-(m+1-k))=2(km+k-1),
\]
as desired. 
\end{proof}

We also need the following lemma, which seems to be well known but  is not so easily pinpointed in the literature in clear form with proof, 
so we provide an argument utilizing the work of Verona on triangulations of stratifications and stratified maps. This is also essentially
\cite[Problem 10.4]{MKdt}, but we prefer not to cite an unworked exercise.

\begin{lemma}\label{triangle}
Let $N$ be a proper smooth submanifold of a smooth manifold $M$.  Then there is a triangulation of $M$ as a $PL$ manifold such that $N$ 
is triangulated as a PL submanifold.
\end{lemma}

Here, ``proper'' means that the embedding $N$ into $M$ is proper, i.e. the intersection of any compact subset of $M$ with $N$ is a compact subset of $N$.

\begin{proof}
We can consider the pair $(M,N)$ to be an abstract (Thom-Mather) stratified space using the tubular neighborhood theorem 
\cite[Theorem 2.3.3]{Wa} to provide the tube structure around $N$. The manifold pair can then be smoothly triangulated by \cite[Theorem 7.8]{V}.
By \cite[Theorem 5]{Wh}, any $C^1$ triangulation of a manifold is a combinatorial (PL manifold) triangulation, in that every 
vertex has a link that is combinatorially equivalent to a sphere. 
\end{proof}

\begin{lemma}\label{orthogonal}
Suppose that $\widetilde{V}_1, \widetilde{V}_2, \dots, \widetilde{V}_n$ are line bundles over $X$ with the property that
$\bigoplus_{i=1}^n \widetilde{V}_i = \Theta^n(X)$.  Endow $\Theta^n(X)$ with its standard Hermitian structure.
Then there exist pairwise orthogonal line bundles 
$V_1, V_2, \dots, V_n$ such that $\bigoplus_{i=1}^n V_i = \Theta^n(X)$ and
 $c_1(V_i) = c_1(\widetilde{V}_i)$ for all $1 \leq i \leq n$.
\end{lemma}

\begin{proof}
Set $V_1 = \widetilde{V}_1$, and suppose that for $k\geq 2$ we have chosen line bundles $V_1, V_2, \dots, V_{k-1}$ such that 
\begin{enumerate}
\item $c_1(V_i) = c_1(\widetilde{V}_i)$ for all $1 \leq i < k$;
\item $\bigoplus_{i=1}^{k-1}V_i = \bigoplus_{i=1}^{k-1}\widetilde{V}_i$;
\item $V_1, V_2, \dots, V_{k-1}$ are pairwise orthogonal.
\end{enumerate}

The bundle $\bigoplus_{i=1}^{k-1}V_i = \bigoplus_{i=1}^{k-1}\widetilde{V}_i$ is a rank $k - 1$ subbundle of $\bigoplus_{i=1}^k\widetilde{V}_i$.
Let $V_k$ be the orthogonal complement of $\bigoplus_{i=1}^{k-1}V_i$ in $\bigoplus_{i=1}^k\widetilde{V}_i$.  Then $V_k$ is a line bundle
that is orthogonal to each $V_i$ with $1 \leq i < k$, and it is immediate that $\bigoplus_{i=1}^kV_i = \bigoplus_{i=1}^k\widetilde{V}_i$.

To check the Chern class conditions, observe by the Whitney product formula that
\[
\sum_{i=1}^k c_1(\widetilde{V}_i) =c_1\left(\bigoplus_{i=1}^k \widetilde{V}_i\right) 
= c_1\left(\bigoplus_{i=1}^k V_i\right) = \sum_{i=1}^k c_1(V_i).
\]
Because $c_1(V_i)=c_1(\widetilde{V}_i)$ for $1 \leq i < k$, it follows that $c_1(V_k)=c_1(\widetilde{V}_k)$. 
The lemma follows by induction.
\end{proof}

\begin{proposition}\label{bundles}
Suppose that $X$ is a $CW$ complex with $\dim(X) \leq 3$, and let $\alpha_1, \alpha_2, \dots, \alpha_n$ be elements
of $H^2(X)$ with $\alpha_1 + \alpha_2 + \cdots + \alpha_n = 0$.  Then there exist line bundles $V_1, V_2, \dots, V_n$
over $X$ such that
\begin{enumerate}
\item $c_1(V_i) = \alpha_i$ for $1 \leq i \leq n$;
\item $\bigoplus_{i=1}^nV_i = \Theta^n(X)$, the trivial bundle of rank $n$;
\item the bundles are pairwise orthogonal.
\end{enumerate}
\end{proposition}

\begin{proof}
By Lemma \ref{orthogonal}, it suffices to show that we can find line bundles $V_1, V_2, \dots, V_n$ over $X$ whose direct
sum is $\Theta^n(X)$ and that have the desired Chern classes.  Suppose we have $n-1$ line bundles $V_1, V_2, \dots, V_{n-1}$ that
span an $n-1$ dimensional subspace of $\mathbb{C}^n$ at each point of $X$, and suppose that $c_1(V_i) = \alpha_i$
for $1 \leq i < n$.  Define $V_n$ to be the orthogonal complement of $\bigoplus_{i=1}^{n-1}V_i$ in $\Theta^n(X)$.  Then
$V_1, V_2, \dots, V_n$ clearly satisfy (2) and (3), and by Lemma \ref{relations}, we have (1) as well.  Hence we need only
construct $V_i$ for $1 \leq i \leq n - 1$.

Next, we use the fact that for each line bundle $V$ over $X$, there exists a continuous map $f: X \rightarrow \mathbb{C}P^\infty$ 
such that $V$ is isomorphic to $f^*\gamma^1$, where $\gamma^1$ denotes the tautological line bundle over $\mathbb{C}P^\infty$
\cite[Theorem 14.6]{MS}.  Note that we can find such a map to
give any desired element of $H^2(X)$ as $c_1(f^*\gamma^1)$: Recall that $\mathbb{C}P^\infty$ is the classifying space both for
complex line bundles over $X$ and for $H^2(X)$, because $X$ is a CW complex (see \cite[Chapter 14]{MS} and
\cite[Example 4.50 and Theorem 4.57]{Ha}). So given any element $\beta\in H^2(X)$, there is a map $g: X\to \mathbb{C}P^\infty$ 
such that $\beta=g^*(\alpha)$ for $\alpha\in H^2(\mathbb{C}P^\infty)$ a generator. Because we also have $\alpha = c_1(\gamma^1)$, 
it follows by the naturality of the Chern classes that $g^*\gamma^1$ has the desired first Chern class.  The map $g$ has codomain
$\mathbb{C}P^\infty$, but by the $CW$ approximation theorem  \cite[Theorem 4.8]{Ha}, because\footnote{Although we are using here
that $\dim(X)\leq 3$, this is not really the critical application of the dimension condition in the proposition. If $\dim(X)$ were of higher 
dimension, we could still make the argument here by imposing an additional condition that $n$ be sufficiently large. The critical use
of $\dim(X)\leq 3$ is yet to come.} $\dim(X) \leq 3$, there is a map  $f:X\to \mathbb{C}P^\infty$ that is homotopic to $g$ and whose 
image is contained in the $3$-skeleton of $\mathbb{C}P^\infty$. The $3$-skeleton of $\mathbb{C}P^\infty$ is $\mathbb{C}P^1$, so the 
desired map to $\mathbb{C}P^{n-1}$ exists so long as $n\geq 2$. But if $n=1$, the hypotheses of the proposition ensure that the only
Chern class vanishes, and so the trivial map $X \rightarrow \mathbb{C}P^0 = \text{pt}$ suffices. Hence, for any class $\beta$ in $H^2(X)$, 
we have constructed a map $f: X \rightarrow \mathbb{C}P^{n-1}$ such that $f^*\gamma^1$ has $\beta$, which was chosen arbitrarily, 
as its first Chern class. 

Now let 
\[
f=(f_1, f_2, \dots, f_{n-1}): X \rightarrow \prod_{i=1}^{n-1}\mathbb{C}P^{n-1}
\]
be a map with the property that $c_1(f_i^*\gamma^1) = \alpha_i$ for each $1 \leq i \leq n - 1$. Then the bundles $V_i = f_i^*\gamma^1$ have the 
desired Chern classes.  We have to make sure that the $V_i(x)$ are linearly independent at each point of $x$. This will not necessarily be the case for 
an arbitrary $f$, but we will show that we can replace $f$ by a homotopic map $g$ that does have this property. Because a homotopy of $f$ 
restricts to a homotopy on each factor, our Chern classes will not be affected by the replacement of $f$ by $g$. 
For $k < n$, let $U_k$ denote the open set  consisting of points $(y_1,y_2, \dots, y_k)$
in $\prod_{i=1}^k \mathbb{C}P^{n-1}$ for which $\Span(y_1, y_2, \dots, y_k)$ has (complex) dimension $k$. 
Suppose we have a map $g: X \rightarrow U_k\times \mathbb{C} P^{n-1}$ for $k < n-2$. 
We will show that $g$ is homotopic to a map into $U_{k+1}$. 

We proceed by induction.  At the base step, we note that $U_1=\mathbb{C}P^1$, so any map into $\mathbb{C}P^1$ is a map
into $U_1$. Suppose now we have a map $g: X \rightarrow U_k\times \mathbb{C}P^{n-1}$ for $1\leq k \leq n - 2$. 
Let $S_{k+1}\subset U_k\times \mathbb{C}P^{n-1}$ consist of those $(y_1, y_2, \dots, y_{k+1})$ in $U_k\times \mathbb{C}P^{n-1}$ such that 
$y_{k+1}$ is in $\Span(y_1, y_2, \dots, y_k)$.  The set $U_{k+1}$ is the complement of $S_{k+1}$ in $U_k\times \mathbb{C}P^{n-1}$, 
and by Lemma \ref{CPmanifold}, we know that $S_{k+1}$ is a smooth submanifold of $U_k\times \mathbb{C}P^{n-1}$
with real dimension $2(nk-1)$.  And because $U_k$ is an open subset of $(\mathbb{C}P^{n-1})^k$,
the open submanifold $U_k\times \mathbb{C}P^{n-1}$ has real dimension $2(k+1)(n-1)=2(nk+n-k-1)$.
Therefore the real codimension of $S_{k+1}$ in $(\mathbb{C}P^{n-1})^{k+1}$ is $2(nk+n-k-1)-2(nk-1)=2(n-k)$. Because we have taken
$k \leq n - 2$, this codimension is greater than or equal to $4$. But $X$ has 
dimension\footnote{Here is where we use the assumption about $\dim(X)$!} 
less than or equal to $3$, so we are able to use a general position argument to homotope $g$ off of $S_{k+1}$.
Here are the technical details of that argument, which are slightly complicated by the fact that $X$ is a CW complex, but not
necessarily a manifold:  Because $U_k \times \mathbb{C}P^{n-1}$ is an open subset of the smooth manifold 
$(\mathbb{C}P^{n-1})^k$, it is itself a smooth manifold, and $S_{k+1}$ is a smooth submanifold. By Lemma \ref{triangle}, there is a smooth 
triangulation of $U_{k+1} \times \mathbb{C}P^{n-1}$ as a PL manifold for which $S_{k+1}$ is triangulated as a subcomplex.  If we choose a triangulation $T$ of
$U_k\times \mathbb{C}P^{n-1}$ \cite[Lemma 3.7]{Hu}, we can treat that triangulation as a CW complex and use the CW approximation theorem
\cite[Theorem 4.8]{Ha} to homotope $g$ to a map with image in the $3$-skeleton of the triangulation. By PL general position \cite[Lemma 4.6]{Hu}, 
we can then perform a homotopy so that the image of the  $3$-skeleton $T^3$  moves into general position with respect to $S_{k+1}$. But due to the dimensions, this means that the image of $T^3$ after the homotopy does not intersect $S_{k+1}$. Combining the homotopies of the 
cellular approximation and the restriction to the image of $X$ of the general position homotopy gives the desired homotopy of $g$.

We now complete the proof of the proposition.  Suppose our map 
\[
f=(f_1, f_2, \dots, f_{n-1})
\]
above has image in $U_k\times \prod_{i=1}^{n-k-1}\mathbb{C}P^{n-1} \subset \prod_{i=1}^{n-1}\mathbb{C}P^{n-1}$.  Then there is a homotopy
that is constant on the last $n-k-2$ coordinates between $f$ and, say, 
$\widetilde{f}=(\widetilde{f}_1,\widetilde{f}_2, \dots, \widetilde{f}_{k+1}, \widetilde{f}_{k+2}, \dots, \widetilde{f}_{n-1})$,
with image in $U_{k+1}\times \prod_{i=1}^{n-k-2}\mathbb{C}P^{n-1}$. By induction, we end up with a map with image in $U_{n-1}$, and this 
is the desired map with the property that the set of lines represented by the coordinates of each image point in 
$(\mathbb{C}P^{n-1})^{n-1}$ spans an $n-1$ dimensional subspace of $\mathbb{C}^n$ and has the desired collection of Chern classes. 
\end{proof}

\begin{proof}[Proof of Theorem \ref{lowdim}]
Choose a continuous ordering $\lambda_1, \lambda_1, \dots, \lambda_n$ of the zeros of $\mu$.   
As we noted in Section \ref{prelims}, given a normal matrix $A$ in $\M(n, C(X))$ with characteristic polynomial $\mu$ satisfying the
hypotheses of the theorem, the obstruction to $A$ being diagonalizable lives in $\bigoplus_{i=1}^n c_1(V_i)$.  
By Lemma \ref{relations}, we see that 
\[
c_1(V_n) = -c_1(V_1) -c_1(V_2) - \cdots -c_1(V_{n-1}),
\]
so it suffices to show we can realize any $(n - 1)$-tuple of elements in $H^2(X)$.  Take elements 
$\alpha_1, \alpha_2, \dots, \alpha_{n-1}$ in $H^2(X)$, and set $\alpha_n = -\alpha_1 - \alpha_2 - \cdots - \alpha_{n-1}$.
By Proposition \ref{bundles}, there exist pairwise orthogonal line bundles $V_1, V_2, \dots, V_n$ over $X$ such
that $c_1(V_i) = \alpha_i$ for $1 \leq i \leq n$.  As in Section \ref{prelims}, define $P_i$ to be the orthogonal projection 
from $\Theta^n(X)$ to $V_i$.  Then
\[
A = \lambda_1P_1 + \lambda_2P_2 + \cdots + \lambda_nP_n
\]
is a normal operator whose obstruction to diagonalizability is $\bigoplus_{i=1}^n\alpha_i$.
\end{proof}

Proposition \ref{CPm} shows that we cannot, in general, extend Theorem \ref{lowdim} to $CW$ complexes with dimension greater than $3$.  
The proof of Proposition \ref{bundles} demonstrates a geometric reason for this dimension limit.

\section{Chern-Weil formulas}
In this final section we show how one can compute the invariant $\theta(A, B)$, up to torsion, when $A$ and $B$ are matrices over a smooth 
manifold $X$ that have split characteristic polynomial.  We do this by employing an operator-algebraic approach to classical Chern-Weil theory.  We refer to the reader to  
\cite[Section 4.3]{P} for complete details, but here are the salient points:  Suppose $V$ is a smooth complex vector bundle over a smooth
manifold $X$.  Embed $V$ in a trivial bundle $\Theta^n(X)$, give $\Theta^n(X)$ its standard Hermitian structure, and let $P \in \M(n, C(X))$ be
the orthogonal projection from $\Theta^n(X)$ to $V$.  Let $dP$ denote the matrix of one--forms obtained by applying the
exterior derivative $d$ to each entry of $P$.  Then $\frac{1}{2\pi i}\tr(PdPdP)$ is a closed two-form whose class $H^2_{de R}(X)$ 
in the de Rham cohomology of $X$ is $c_1(V)$ (the scaling factor $\frac{1}{2\pi i}$, which is not present in \cite{P}, is to give us
a class with integer coefficients).

Suppose $A$ is a normal multiplicity-free matrix over a smooth manifold $X$, that the characteristic polynomial of $A$ splits over
$C(X)$, and that the matrix entries of $A$ are smooth functions.  As usual, we write
\[
A = \lambda_1P_1 + \lambda_2P_2 + \cdots + \lambda_n P_n,
\]
with $P_i$ the orthogonal projection onto the line subbundle $V_i$ of $\Theta^n(X)$ that is the $\lambda_i$-eigenbundle.  Then for
each $1 \leq i \leq n$, the cohomology class $\left[\frac{1}{2\pi i}\tr(P_idP_idP_i)\right]$ 
is a diagonalizability obstruction in $H^2_{de R}(X)$.

\begin{example} The diagonalizabilty obstruction of a  matrix over $S^2$
\end{example}\label{onematrix}

Take $S^2$ to be the unit sphere in $\mathbb{R}^3$ and define
\[
A = \begin{pmatrix}
x^2 + x^3 + y^2 + xy^2 + i(1 - x)z^2& (y + iz)(x^2 + y^2 - iz^2) \\
(y - iz)(x^2 + y^2 - iz^2) & x^2 - x^3 + y^2 - xy^2  + i(1 + x)z^2
\end{pmatrix}.
\]
It is straightforward to check that $A$ is normal and that $\tr A  = 2(x^2 + y^2 + iz^2)$.  To compute the determinant of $A$,
 substitute $1 - x^2 - y^2$ for $z^2$; an involved computation yields
\begin{multline*}
\det A = -1 + (3 + 2i)x^2 - 2x^4 - 2ix^6 + (3 + 2i)y^2 \\ - 4x^2y^2 - 6ix^4y^2   - 2y^4  
- 6ix^4y^2 - 2iy^6 + z^2 \\
-(2 - 2i)x^2z^2 - 2ix^4z^2  - (2 - 2i)y^2z^2 - 4ix^2y^2z^2 - 2iy^4z^2.
\end{multline*}
The real part of this expression is 
\[
-1 + 3x^2 - 2x^4 + 3y^2 - 4x^2y^2 - 2y^4 + z^2 - 2x^2z^2 - 2y^2z^2;
\]
substituting $1 - x^2 - y^2$ for $z^2$ again, we find that the real part vanishes.  The imaginary part of $\det A$ is 
\[
2x^2 - 2x^6 + 2y^2 - 6x^4y^2 - 6x^2y^4  - 2y^6 +2x^2z^2 - 2x^4z^2 + 2y^2z^2 - 4x^2y^2z^2 - 2y^4z^2,
\]
and if we once again replace $z^2$ by $1 - x^2 - y^2$, the imaginary part of the determinant simplifies to
\[
4(x^2 - x^4 + y^2 - y^4 - 2x^2y^2)  = 4(x^2 + y^2)(1 - x^2 - y^2) = 4(x^2 + y^2)z^2.
\]
Therefore $A$ has characteristic polynomial
\begin{align*}
\mu_A(\lambda) &= \lambda^2 - 2(x^2 + y^2 + iz^2)\lambda + 4i(x^2 + y^2)z^2 \\
&= \Bigl(\lambda - 2(x^2 + y^2)\Bigr)\Bigl(\lambda - 2iz^2\Bigr).
\end{align*}
Set $\lambda_1(x, y, z) = 2x^2 + 2y^2$ and $\lambda_2(x, y, z) = 2iz^2$.  Using formula \eqref{E: projection equation} from Section \ref{prelims},
\[
P_1 = \frac{1}{2x^2 + 2y^2 - 2iz^2}\cdot(A - 2iz^2I)
=  \frac{1}{2}\begin{pmatrix} 1 + x & y + iz \\ y - iz & 1 - x\end{pmatrix}.
\]

To simplify our computations, convert to polar coordinates:
\[
P_1 = \frac{1}{2}\begin{pmatrix} 1 + \sin\phi\cos\theta & \sin\phi\sin\theta + i\cos\phi \\
\sin\phi\sin\theta - i\cos\phi & 1 - \sin\phi\cos\theta\end{pmatrix}.
\]

We have
\[
dP_1 = \frac{1}{2}\begin{pmatrix} \omega_{11} & \omega_{12} \\ \omega_{21} & \omega_{22}\end{pmatrix},
\]
with
\begin{align*}
\omega_{11} &= d(1 + \sin\phi\cos\theta) = -\sin\phi\sin\theta\,d\theta + \cos\phi\cos\theta\,d\phi \\
\omega_{12} &= d(\sin\phi\sin\theta + i\cos\phi) = \sin\phi\cos\theta\,d\theta + (\cos\phi\sin\theta - i\sin\phi)\,d\phi \\
\omega_{21} &= d(\sin\phi\sin\theta + i\cos\phi) = \sin\phi\cos\theta\,d\theta + (\cos\phi\sin\theta + i\sin\phi)\,d\phi \\
\omega_{22} &= d(1 + \sin\phi\cos\theta) = \sin\phi\sin\theta\,d\theta - \cos\phi\cos\theta\,d\phi,
\end{align*}
Direct computation gives us
\[
dP_1dP_1 = \frac{1}{2} \begin{pmatrix}
i\sin^2\phi\cos\theta & i\sin^2\phi\sin\theta - \sin\phi\cos\phi \\
 i\sin^2\phi\sin\theta + \sin\phi\cos\phi & -i\sin^2\phi\cos\theta \end{pmatrix}d\theta d\phi,
\]
and 
\[
P_1dP_1dP_1 = \frac{1}{4}\begin{pmatrix} f_{11} & f_{12} \\ f_{21} & f_{22}\end{pmatrix}d\theta d\phi,
\]
with
\begin{align*}
f_{11} &= (1 + \sin\phi\cos\theta)(i\sin^2\phi\cos\theta) + (\sin\phi\sin\theta + i\cos\phi)(i\sin^2\phi\sin\theta + \sin\phi\cos\phi) \\
&= i\sin\phi(1 + \sin\phi\cos\theta)
\end{align*}
and
\begin{align*}
f_{22} &= (\sin\phi\sin\theta - i\cos\phi)(i\sin^2\phi\sin\theta - \sin\phi\cos\phi) + (1 - \sin\phi\cos\theta)(-i\sin^2\phi\cos\theta)\\
&= i\sin\phi(1 - \sin\phi\cos\theta).
\end{align*}
Therefore
\[
\tr(P_1dP_1dP_1) = \frac{1}{4}\bigl(i\sin\phi(1 + \sin\phi\cos\theta) + i\sin\phi(1 - \sin\phi\cos\theta)\bigr)d\theta d\phi
= \frac{1}{2}i\sin\phi\, d\theta d\phi.
\]
The class $\bigl[\frac{1}{2\pi i}\cdot\frac{1}{2}i\sin\phi\, d\theta d\phi\bigr] = 
\bigl[\frac{1}{4\pi}\sin\phi\, d\theta d\phi\bigr]$ is nontrivial in $H^2_{de R}(S^2)$, because
\[
\int_0^\pi\int_0^{2\pi}\frac{1}{4\pi}\sin\phi\, d\theta d\phi = 1 \neq 0,
\]
and thus we see that $A$ is not diagonalizable.    

We can do a similar computation to compute $[\frac{1}{2\pi i}\tr(P_2 dP_2dP_2)]$ for 
\[
P_2 = \frac{1}{2iz^2 - (2x^2 + 2y^2)}\cdot(A - (2x^2 + 2y^2)I)
=  \frac{1}{2}\begin{pmatrix} 1 - x & -y - iz \\ -y + iz & 1 + x\end{pmatrix},
\]
but it is simpler to use the fact that the sum of the first Chern classes of the eigenvalue bundles of $A$ is zero, so
\[
\left[\frac{1}{2\pi i}\tr(P_2 dP_2dP_2)\right] = - \left[\frac{1}{2\pi i}\tr(P_1 dP_1dP_1)\right] 
= -\left[\frac{1}{4\pi}\sin\phi\, d\theta d\phi\right]
\]
in $H^2_{de R}(S^2)$.
\newline

Now suppose $A$ and $B$ are multiplicity-free normal matrices over $X$ that have the same characteristic polynomial $\mu$,
and suppose $\mu$ splits over $C(X)$.  Then we can construct a Chern-Weil-type formula for $\theta(A, B)$.  To do this, we first
suppose that $V$ and $W$ are one-dimensional vector subspaces of $\mathbb{C}^n$.  Equip $\mathbb{C}^n$ with its standard orthonormal basis,
and let $P$ and $Q$ be the matrices that represents the orthogonal projection of $\mathbb{C}^n$ onto $V$ and $W$ respectively.  
Define $\Psi: \M(n, \mathbb{C}) \rightarrow \M(n, \mathbb{C})$ by the formula
\[
\Psi(T) = QTP.
\]
Note that $\Psi$ is the projection of all linear transformations of $\mathbb{C}^n$ onto linear maps from $\Ran P = V$ to $\Ran Q = W$.
Next, for $1 \leq i, j \leq n$, let $e_{ij}$ denote the matrix that is $1$ in the $(i, j)$ slot and zero elsewhere.  We can 
make the vector space $\M(n, \mathbb{C})$ into a finite dimensional Hilbert space by decreeing 
\[
e_{11}, e_{12}, \dots, e_{1n}, e_{21}, e_{22}, \dots, e_{2n}, e_{31}, \dots, e_{n,n-1}, e_{nn}
\]
to be an orthonormal basis.  Careful bookkeeping yields that for each $1 \leq k, \ell \leq n$, 
\[
\Psi(e_{k\ell}) = \sum_{i,j=1}^nq_{ik}p_{\ell j}e_{ij}.
\]
In terms of our ordered orthonormal basis, the linear map $\Psi$ has matrix
\[
\begin{pmatrix}
q_{11}p_{11} &\cdots & q_{11}p_{n1} &\cdots & q_{1n}p_{11} & \cdots & q_{1n}p_{n1} \\
\vdots & \ddots & \vdots &  & \vdots & \ddots & \vdots \\
q_{11}p_{1n} &\cdots & q_{11}p_{nn} &\cdots & q_{1n}p_{1n} & \cdots & q_{1n}p_{nn} \\
\vdots & & \vdots &  & \vdots &  & \vdots \\
q_{n1}p_{11} &\cdots & q_{n1}p_{n1} &\cdots & q_{1n}p_{11} & \cdots & q_{1n}p_{n1} \\
\vdots & \ddots & \vdots &  & \vdots & \ddots & \vdots \\
q_{n1}p_{1n} &\cdots & q_{n1}p_{nn} &\cdots & q_{nn}p_{1n} & \cdots & q_{nn}p_{nn}
\end{pmatrix}.
\]
Letting $P^T$ denote the transpose of $P$, we can write our matrix for $\Psi$ in block form:
\[
\begin{pmatrix}
q_{11}P^T & q_{12}P^T & \cdots & q_{1n}P^T \\
q_{21}P^T & q_{22}P^T & \cdots & q_{2n}P^T \\
\vdots & \vdots & \ddots & \vdots \\
q_{n1}P^T & q_{n2}P^T & \cdots & q_{nn}P^T
\end{pmatrix}.
\]
This is called the \emph{Kronecker product} of $Q$ and $P^T$, and is typically written $Q \otimes P^T$.
\vskip 6pt

Now suppose that $V$ and $W$ are complex line bundles over $X$ and that $P$ and $Q$ are the projections
of $\Theta^n(X)$ to $V$ and $W$, respectively.  Then we can define $R := Q \otimes P^T$ pointwise.  Using 
easily-verified properties of the Kronecker product and transpose \cite[Section 2.3]{G}, we check that
\[
R^2 = (Q \otimes P^T)(Q \otimes P^T) = (Q^2 \otimes (P^T)^2) = (Q^2 \otimes (P^2)^T) = Q \otimes P^T = R
\]
and
\[
R^* = (Q \otimes P^T)^* = Q^* \otimes (P^T)^* = Q^* \otimes (P^*)^T = Q \otimes P^T = R.
\]
Thus $R$ is a projection, and by construction is the projection from $\Theta^n(X)$ onto $\Hom(V, W)$.  Therefore
\[
c_1(\Hom(V, W)) =  \left[\frac{1}{2\pi i}\tr(R\, dR\, dR)\right]
\]
in de Rham cohomology.

\begin{example} The invariant for a pair of matrices over $S^2$
\end{example}
Let $A = \lambda_1P_1 + \lambda_2P_2$ be the matrix from Example 5.1, and define
\[
B = \begin{pmatrix} 
x^2 - x^2z + y^2 - y^2z + iz^2(z + 1) & (x + iy)(ix^2 + iy^2 + z^2) \\
 (-x + iy)(ix^2 + iy^2 + z^2) & x^2 + x^2z + y^2 + y^2z + iz^2(z - 1) \end{pmatrix}.
\]
Then $B$ is a normal matrix over $S^2$, and by computations similar to those we performed for the matrix $A$ in Example 5.1,
we obtain
\[
\mu_B(\lambda) = \lambda^2 - 2(x^2 + y^2 + iz^2)\lambda + 4i(x^2 + y^2)z^2 = \mu_A(\lambda).
\]
Write $B = \lambda_1Q_1 + \lambda_2Q_2$.  Then 

\[
Q_1 = \frac{1}{2}\begin{pmatrix} 1 - z & -y + ix \\ -y - ix & 1 + z\end{pmatrix},
\]
and the Kronecker product of $Q_1$ and the transpose of $P_1$ is the matrix

\[
R_1= \frac{1}{4}
\begin{pmatrix}
(1 - z)(1 + x) & (1 - z)(y - iz) & (-y + ix) (1 + x) & (-y + ix)(y - iz) \\
(1 - z)(y + iz) & (1 - z)(1 - x) & (-y + ix)(y + iz) & (-y + ix)(1 - x) \\
(-y - ix)(1 + x) & (-y - ix)(y - iz) & (1 + z)(1 + x) & (1 + z)(y - iz) \\
(-y - ix)(y + iz) & (-y - ix)(1 - x) & (1 + z)(y + iz)& (1 + z)(1 - x)
\end{pmatrix}
\]
Using \cite{Bo} and converting to polar coordinates, we obtain
\[
\tr(R_1\, dR_1\, dR_1) = i(z\, dx\, dy - y\, dx\, dz + x\, dy\, dz) = -i\sin\phi\, d\theta\, d\phi.
\]
Thus
\[
\int_{S^2} \frac{1}{2\pi i}\tr(R_1\, dR_1\, dR_1) =
\frac{1}{2\pi i}\int_0^\pi\int_0^{2\pi}-i\sin\phi\, d\theta d\phi = -2 \neq 0,
\]
and therefore $A$ and $B$ are not unitarily equivalent.  Similarly, one can show that 
\[
\tr(R_2\, dR_2\, dR_2) = i\sin\phi\, d\theta\, d\phi.
\]

\end{document}